\documentclass[12pt]{article}
\usepackage{graphicx} 
\usepackage{hyperref}
\usepackage{amsmath,amssymb,amsfonts,amsthm,tikz,xspace,fullpage}
\usepackage{authblk}
\usetikzlibrary{calc}
\usetikzlibrary{arrows.meta}

\title{On the Burning Game: Nordhaus–Gaddum Bounds and Graph Products}
\author[1]{Nina Chiarelli}
\author[2,4]{Vesna Iršič Chenoweth}
\author[3,4]{Marko Jakovac}
\author[5]{William B. Kinnersley}
\author[6]{Mirjana Mikala\v{c}ki} 
\affil[1]{FAMNIT and IAM, University of Primorska, Slovenia}
\affil[2]{Faculty of Mathematics and Physics, University Ljubljana, Slovenia}
\affil[3]{Faculty of Natural Sciences and Mathematics, University of Maribor, Slovenia}
\affil[4]{Institute of Mathematics, Physics and Mechanics, Ljubljana, Slovenia}
\affil[5]{Department of Mathematics and Applied Mathematical Sciences, University of Rhode Island, USA}
\affil[6]{Department of Mathematics and Informatics, Faculty of Sciences, University of Novi Sad, Serbia}

\date{}

\newtheorem{theorem}{Theorem}
\newtheorem{conjecture}[theorem]{Conjecture}
\newtheorem{corollary}[theorem]{Corollary}
\newtheorem{lemma}[theorem]{Lemma}

\newtheorem{problem}[theorem]{Problem}

\newtheorem{proposition}[theorem]{Proposition}

\theoremstyle{definition}
\newtheorem{example}[theorem]{Example}

\newcommand{\bg}{b_{\rm g}}
\newcommand{\ceil}[1]{\left\lceil #1 \right\rceil}
\newcommand{\floor}[1]{\left\lfloor #1 \right\rfloor}

\newcommand{\cart}{\, \Box \,}
\newcommand{\strongprod}{\, \boxtimes \,}
\newcommand{\fastplayer}{Burner\xspace}
\newcommand{\slowplayer}{Staller\xspace}
\DeclareMathOperator{\rad}{rad}
\DeclareMathOperator{\diam}{diam}
\DeclareMathOperator{\CL}{CL}

\begin{document}

 \maketitle

\begin{abstract}
    We continue research on the burning game on graphs. Given a graph $G$, two players, \fastplayer and \slowplayer, take turns in selecting vertices of $G$ to burn. All burned vertices spread fire to unburned neighboring vertices, as in the burning process. The goal of \fastplayer is to burn the graph as quickly as possible, while \slowplayer wants the process to last as long as possible. If both players play optimally, then the number of time steps needed to burn the whole graph $G$ is the game burning number $\bg(G)$ if \fastplayer makes the first move, and the \slowplayer-start game burning number $\bg'(G)$ if \slowplayer starts.

    In this paper, we study this game further, establishing Nordhaus-Gaddum bounds on the game burning number,
    as well as bounds for four different types of graph products: strong, Cartesian, lexicographic and corona products.

\end{abstract}

\section{Introduction}

The burning process was introduced by Bonato et al.~in~\cite{bonato2016burn} with the aim of modeling the spread of viruses in both medical and computer networks and also the spread of various trends and influences over all types of networks, which was previously substantially researched (see, e.g.~\cite{banerjee2014epidemic, kephart1992directed, kramer2014experimental, rogers2014diffusion}). As a dual of the burning process, the \textit{cooling process} was introduced recently in~\cite{bonato2024cool} by Bonato et al., modeling the diminution of propagation in networks. 

In both burning and cooling processes, the setup is the same (see~\cite{bonato2021burning, bonato2022invitation, bonato2016burn, bonato2024cool}): they are discrete-time processes on a given finite, simple, undirected graph $G$, in which each vertex can be in one of the two states - burning or unburned. Initially, at time step $t = 0$, all vertices are unburned. 
In each subsequent time step -- called \emph{a round} -- every burning vertex spreads fire to all of its neighbors (causing them to burn), after which one of the unburned vertices (if there is one) is selected to burn; we call this vertex the \emph{source} for the round. Once burning, a vertex stays in this state until the end of the game. 
The process ends when all vertices in the graph are burned. 
The main difference between burning and cooling processes is in the goal of the process. In the burning process, the aim is that the graph burns as quickly as possible, i.e.\ to minimize the number of rounds, while in the cooling process, the aim is to maximize the number of rounds. The minimum number of rounds needed for the burning process to finish on a graph $G$ is the \emph{burning number}, denoted by $b(G)$; the maximum number of rounds 
is called \emph{the cooling number} and denoted by $\CL(G)$. For all graphs $G$, it holds that $b(G)\leq \CL(G)$. 

Motivated by both processes, the authors recently introduced 
the \textit{burning game} (see~\cite{burning-1}) wherein the two players, \fastplayer (he/him) and \slowplayer (she/her) take turns selecting 
the sources in each round of the burning process.  More formally, the game is defined as follows. In a given simple, finite graph $G$, at the beginning of the game, all vertices are unburned. This is the time step $t=0$, i.e. the round $0$. In all other time steps $t\geq 1$, first, in the \textit{spreading phase}, all unburned neighbors of burning vertices begin to burn, and then, in the \textit{selection phase}, one of the players chooses an unburned vertex to burn. 
So, each round consists of these two phases, but note 
that it may happen that the last round does not have the selection phase (since there may not be any remaining unburned vertices). The game ends in the first round $t$ in which all vertices of $G$ are burning. The aim of \fastplayer is to minimize the total number of time steps and the aim of \slowplayer is to maximize it. If both players play optimally, then the number of rounds needed to burn the whole graph $G$ is called the \emph{game burning number} $\bg(G)$ if \fastplayer makes the first move, and the \emph{\slowplayer-start game burning number} $\bg'(G)$ if \slowplayer starts. As mentioned in~\cite{burning-1}, $b(G) = \bg(G)$ when \fastplayer is the only player (with possible slight difference in the last round, as in the game, the spread comes before selection, but in the original burning process, the two may happen simultaneously), while $\CL(G) = \bg(G)$ when \slowplayer is the only player. 

Graph burning has attracted a lot of attention recently, and much of it has been focused on resolving the so-called \textit{burning number conjecture} posed by Bonato et al.\ in \cite{bonato2016burn}, which asserts that every $n$-vertex connected graph $G$ satisfies $b(G) \le \ceil{\sqrt{n}}$. There has been a substantial improvement of the upper bounds on $b(G)$ over time (see e.g. \cite{BBBCKPR23,BBJRR18,LL16,NT24}), and also knowing that if the burning conjecture holds for trees, then it holds for all connected graphs, the central point of the research is exactly the class of trees, i.e. some of the families of trees (\cite{BL19,DDSSS18,HTK20,LHH20}).  

In~\cite{burning-1}, basic bounds on $b_{\rm g}(G)$ are given and several fundamental properties of the burning game are established. Graphs with small game burning numbers are characterized and the game is studied on paths and cycles. An analogue of the burning number conjecture for the burning game is also considered. Finally, it is shown that the problem of determining whether or not $b_{\rm g}(G) \le k$ is NP-hard.  

In this paper, we continue the study of this game, 
establishing Nordhaus-Gaddum bounds on the game burning number, i.e. bounds on the sum and product of the game burning number and the \slowplayer-start game burning number of the given graph and its complement. We further look at several different types of graph products -- Cartesian, strong, lexicographic, and corona products -- and establish bounds for both $\bg$ and $\bg'$.

\subsection{Organization of the paper} The remainder of the paper is organized as follows. In the following subsection, we state some notation that will be used throughout the rest of the paper. Then in Section~\ref{prel}, we recall some results from~\cite{burning-1} that are relevant for our new results further on. In Section~\ref{sec:nordhaus-gaddum}, we establish lower and upper bounds on the sum and product of $\bg$ and $\bg'$ for a given graph and its complement. In Section~\ref{sec:products}, we give bounds on $\bg$ and $\bg'$ for four types of graph products: Cartesian, strong, lexicographic and corona products. Finally, we conclude our paper with some open questions in  Section~\ref{sec:questions}.
 
\subsection{Notation} 
We use standard graph theory notation throughout. For a given graph $G$, we denote its vertex set and edge set by $V(G)$ and $E(G)$, respectively, and we set $v(G) = |V(G)|$ and $e(G) = |E(G)|$.  An edge joining two vertices $x$ and $y$ is denoted by $xy$.  The complement of the graph is denoted by $\overline{G}$.
Given graphs $H, F$, we write $H\subseteq F$ to denote that $H$ is contained in $F$, meaning that $V(H)\subseteq V(F)$ and $E(H)\subseteq E(F)$. 
Given any set $S\subseteq V(G)$, we denote the (joint) \textit{external neighborhood} of $S$ by $N_G(S)=\{u \in V(G)\setminus S : \exists x \in S \ \textrm{such that} \ ux \in E(G)\}$; when the graph $G$ is clear from context, we simply write $N(S)$. The \textit{closed neighborhood} of $S$ in $G$ is $N_G(S) \cup S$, and is denoted by $N_G[S]$ (or just $N[S]$).  When $S$ consists of only one vertex, $x$, we write $N_G(x)$ (or $N(x)$) for the external neighborhood and $N_G[x]$ (or $N[x]$) for the closed neighborhood. We let $d_G(x) = |N_G(x)|$ denote the degree of vertex $x$ in graph $G$; once again, when $G$ is clear from context, we simply write $d(x)$. The maximum and minimum vertex degree in a given graph $G$ are denoted by $\Delta(G)$ and $\delta(G)$, respectively. %
For a vertex $v$ and a non-negative integer $k$, the \textit{$k$th closed neighborhood} $N_k[v]$ of $v$ is defined as the set of all vertices within distance $k$ of $v$, including $v$ itself. 
The distance between vertices $u$ and $v$ is denoted by $d(u, v)$. If $G$ is a graph and $u$ is a vertex of $G$, then the \textit{eccentricity} of $u$ is defined as $ecc(u)= \max\{d(u,v): v\in V(G)\}$. The \textit{radius} and \textit{diameter} of $G$ are defined as the minimum and maximum eccentricities, respectively, over all vertices in $G$.
We denote with $G^2$ the \textit{square} of a graph $G$, i.e., the vertex set of $G^2$ is $V(G^2)=V(G)$ and two distinct vertices $u,v \in V(G^2)$ are adjacent if and only if $d(u,v) \le 2$.

\subsection{Preliminaries}
\label{prel}
In this section we list several results about the burning game from \cite{burning-1} as they are needed later in the paper.

The following proposition appears in \cite{burning-1}; we have provided the proof here because we use some of the ideas later in the paper.

\begin{proposition}[{\cite[Proposition 1]{burning-1}}]
    \label{prop:trivial-bounds}
    If $G$ is a connected graph, then $b(G) \leq \bg(G) \leq \min\{ \CL(G), 2 b(G^2) - 1\}$ and $b(G) \leq \bg'(G) \leq \min\{ \CL(G), 2 b(G^2)\}$.
\end{proposition}

\begin{proof}
    Regardless of who plays first, let $x_1, x_2, \ldots, x_k$ be a sequence of moves in the burning game. Selecting $x_1, x_2, \ldots$ as sources in the graph burns the whole graph in $\bg(G)$ number of rounds, thus $\bg(G) \geq b(G)$; similarly $\bg(G) \leq \CL(G)$. 

    To see that $\bg(G) \leq 2 b(G^2) - 1$, consider the following strategy for Burner. Let $b(G^2) = k$ and let $x_1, \dots, x_k$ be an optimal sequence of sources for the burning process on $G^2$. Burner's strategy in the burning game is to play vertex $x_i$ on his $i$th turn or, if $x_i$ is already burned, to play any unburned vertex. By choice of the $x_i$, every vertex $v$ in $G$ must be within distance $k-i$ of some $x_i$ in $G^2$, thus $v$ must be within distance $2(k-i)$ of some $x_i$ in $G$.  Burner's strategy ensures that vertex $x_i$ is burned no later than round $2i-1$; since the fire will reach $v$ within the next $2(k-i)$ rounds of the game, $v$ will burn no later than round $2 k - 1$.

    A similar argument shows that $\bg'(G) \le 2b(G^2)$; the only change is that Burner plays $x_i$ in round $2i$ (provided that it is not already burned by then).
\end{proof}

\begin{proposition}[{\cite[Proposition 2]{burning-1}}]
    \label{prop:radius}
    If $G$ is a connected graph, then $\bg(G) \leq \rad(G) + 1$ and $\bg'(G) \leq \min\{\rad(G) + 2, \diam(G)+1\}$.
\end{proposition}

\begin{proposition}[{\cite[Proposition 3]{burning-1}}]
    \label{prop:Delta}
    Let $G$ be a graph on $n$ vertices. If $\Delta(G) \leq n-2$, then $\bg(G) \leq n - \Delta(G)$, and if $\Delta(G) \leq n-3$, then $\bg'(G) \leq n - \Delta(G)$.
\end{proposition}

\begin{lemma}[{\cite[Lemma 8]{burning-1}}]
    \label{lem:spanning-subgraph}
    If $H$ is a spanning subgraph of $G$, then $\bg(G) \leq \bg(H)$ and $\bg'(G) \le \bg'(H)$.  
\end{lemma}

\begin{proposition}[{\cite[Proposition 9]{burning-1}}]
    \label{prop:characterize-1-2}
    Let $G$ be a connected graph.
    \begin{enumerate}
        \item $\bg(G) = 1$ if and only if $G = K_1$;
        \item $\bg'(G) = 1$ if and only if $G = K_1$;
        \item $\bg(G) = 2$ if and only if $G \neq K_1$ and $\Delta(G) \geq |V(G)| - 2$; and
        \item $\bg'(G) = 2$ if and only if $G \neq K_1$ and $\delta(G) \geq |V(G)| - 2$ (if and only if $G$ is isomorphic to a complete graph on at least two vertices without a (possibly empty) matching).
    \end{enumerate}
\end{proposition}

Finally, we recall a useful general-purpose lemma from \cite{burning-1}.  We sometimes need to jump into the ``middle'' of an instance of the burning game, where some vertices are already burning.  Given a graph $G$ and $B \subseteq V(G)$, we let $G \vert B$ denote the graph $G$, with the vertices in $B$ already burning.  When we play the burning game on $G \vert B$, we say that we are playing the game \textit{relative to $B$}; the number of rounds needed to burn all of $G \vert B$ -- assuming that both players play optimally -- is denoted by $\bg(G \vert B)$ if Burner plays first and by $\bg'(G \vert B)$ if Staller plays first.

Intuitively, beginning the game with some vertices already burned cannot increase the length of the game; the \textit{Continuation Principle}, stated below, formalizes this intuition.

\begin{theorem}[Continuation Principle; {\cite[Theorem 4]{burning-1}}]
    \label{thm:continuation-principle}
    If $A \subseteq B \subseteq V(G)$, then $\bg(G|B) \leq \bg(G|A)$ and $\bg'(G|B) \leq \bg'(G|A)$.
\end{theorem}

\section{Nordhaus–Gaddum type results}
\label{sec:nordhaus-gaddum}

In this section we prove upper and lower bounds for the sum and product of the game burning number of a graph and its complement. The analogous question for the chromatic number was studied in the 1950s by Nordhaus and Gaddum \cite{nordhaus1956NG} and has since been studied for many other graph parameters; see for example \cite{aouchiche2013NG}. This type of result was also studied for graph burning by Bonato et al.~\cite[Section 3]{bonato2016burn}.  In particular, they showed that for any $n$-vertex graph $G$ with $n \ge 2$, we have $4 \le b(G) + b(\overline{G}) \le n+2$ and $4 \le b(G)b(\overline{G}) \le 2n$; if additionally both $G$ and $\overline{G}$ are connected, then $b(G)b(\overline{G}) \le n+6$. 

The mentioned bounds established by Bonato et al.\ for $b(G) + b(\overline{G})$ in fact apply to $\bg(G) + \bg(\overline{G})$ as well, with essentially the same proof; see \cite[Theorem 18]{bonato2016burn}.

\begin{theorem}
\label{thm:sum_gen}
If $G$ is a graph of order $n$, $n\geq 2$, then $$4\le \bg(G) + \bg(\overline{G})\le n+2.$$    
\end{theorem}


Our bounds on $\bg(G)\bg(\overline{G})$ are also identical to the bounds on $b(G)b(\overline{G})$ from \cite{bonato2016burn}; however, establishing them takes a bit more work.  We begin with a helpful lemma.

\begin{lemma}
    \label{lem:disconnected-3}
    If $G$ is a disconnected graph of order $n$ and all connected components of $G$ have order at least 3, then $\bg(G) \leq \frac{n+1}{2}$ and $\bg'(G) \leq \frac{n+1}{2}$. 
\end{lemma}

\begin{proof}
We give a strategy for Burner to ensure that, for most of the game, at least four vertices are burned within each pair of consecutive rounds. On each of Burner's turns, if any unburned vertex $v$ has two or more unburned neighbors, then Burner burns $v$. This ensures that at least two more vertices will be burned in the spreading phase of the subsequent round. If the game does not end during this spreading phase, then one more vertex is chosen by Staller, for a total of at least four vertices burned within two rounds.

Suppose instead that, on Burner's turn, every unburned vertex has at most one unburned neighbor. Since every component of $G$ has order at least $3$, it follows that every component of $G$ has at least one burning vertex. Indeed, otherwise a completely unburned component, being connected and of order at least $3$, would contain a vertex with at least two unburned neighbors. Moreover, each unburned vertex either has a burning neighbor, or has a unique unburned neighbor which itself has a burning neighbor. Consequently, the game will end within the next two rounds, regardless of how the players play.

We first consider the Burner-start game. Let $k$ be the number of Burner's first turns on which he can burn a vertex with at least two unburned neighbors, before either the game ends or no such vertex is available on his next turn. If the game ends by the end of round $2k$, then at least $4k-1$ vertices have burned: the first $k-1$ pairs of rounds account for at least $4(k-1)$ vertices, while Burner's $k$th move and the subsequent spreading phase account for at least three more. Hence $n \geq 4k-1$, and therefore $\bg(G) \leq 2k = \frac{4k-1}{2}+\frac{1}{2} \leq \frac{n+1}{2}$.

Suppose now that the game does not end by the end of round $2k$. It follows that at least $4k$ vertices have burned during the first $2k$ rounds. In addition, at least one vertex must burn during the spreading phase of round $2k+1$. Indeed, if no vertex were burned during this spreading phase, then every component containing an unburned vertex would be completely unburned, and hence would contain a vertex with at least two unburned neighbors, contrary to the choice of $k$.

If the game ends in round $2k+1$, then at least $4k+1$ vertices have burned, hence $\bg(G) \leq 2k+1 = \frac{4k+1}{2}+\frac{1}{2} \leq \frac{n+1}{2}$. If instead the game ends in round $2k+2$, then Burner must have burned another vertex in round $2k+1$, and at least one more burns during the spreading phase of round $2k+2$; hence $\bg(G) \leq 2k+2 = \frac{4k+3}{2}+\frac{1}{2} \leq \frac{n+1}{2}$. Finally, if the game ends in round $2k+3$, then Staller must have burned another vertex in round $2k+2$, and at least one more must have burned in the spreading phase of round $2k+3$, so $\bg(G) \leq 2k+3 = \frac{4k+5}{2}+\frac{1}{2} \leq \frac{n+1}{2}$.

In the Staller-start game, after the first move of Staller, Burner uses the same strategy as above. Since $G$ is disconnected, there is a component which is still completely unburned after Staller's first move. As this component has order at least $3$, it contains a vertex with at least two unburned neighbors, so Burner can use the prescribed strategy on his first turn. Moreover, at least one vertex is burned in the spreading phase of round $2$. Thus, if Burner can burn a vertex with at least two unburned neighbors for his first $k$ turns and the game does not end during the spreading phase following his $k$th move, then after the first $2k+1$ rounds at least $4k+2$ vertices are burned. If the game ends during that spreading phase, then at least $4k+1$ vertices are burned in $2k+1$ rounds, which also gives $\bg'(G)\leq\frac{n+1}{2}$. With similar arguments as above we obtain $\bg'(G)\leq\frac{n+1}{2}$.
\end{proof}

\begin{theorem}
    \label{thm:product_gen}
    If $G$ is a graph of order $n$ with $n \geq 2$, then $$4 \leq \bg(G) \bg(\overline{G}) \leq 2n.$$ 
\end{theorem}

\begin{proof}
    As $n \geq 2$, $\bg(G) \bg(\overline{G}) \geq 4$ and this sets the lower bound. 
    
    For the upper bound, if $\Delta(G) \geq n-2$ then by Proposition~\ref{prop:characterize-1-2} we have $\bg(G) = 2$, and since clearly $\bg(\overline{G}) \leq n$ we get $\bg(G) \bg(\overline{G}) \leq 2n$. 
    
    If $\Delta(G) \leq n-3$, then $\delta(\overline{G})\geq 2$ implies that each component of $\overline{G}$ is of order at least $3$ and that $n\geq 4$.
    Suppose first that $\overline{G}$ is not connected and let $u, v\in V(G)$. If $u$ and $v$ are in different components of $\overline{G}$, then $uv \in E(G)$ and $d_G(u, v) = 1$. On the other hand, if $u, v$ are in the same component of $\overline{G}$, then in $G$ they are both adjacent to some vertex in a different component of $\overline{G}$. Thus $d_G(u, v) \leq 2$ and we have that $\diam(G) \leq 2$. Proposition \ref{prop:radius} implies that $\bg(G) \leq 3$ and Lemma \ref{lem:disconnected-3} gives $\bg(\overline{G}) \leq \frac{n+1}{2}$. Thus $\bg(G) \bg(\overline{G}) \leq 3 \frac{n+1}{2} \leq 2n$ for all $n \geq 3$.
   
   Suppose now that $G$ and $\overline{G}$ are both connected.
   Without loss of generality we may assume that $\bg(G) \geq \bg(\overline{G})$. If $\diam(\overline{G}) \geq 3$, then $\diam(G) \leq 3$, thus $\bg(G) \leq 4$ by Proposition \ref{prop:radius}. By the assumption this means that $\bg(\overline{G}) \leq 4$ as well, so $\bg(G) \bg(\overline{G}) \leq 16 \leq 2n$ for every $n \geq 8$. 
    Otherwise, $\diam(\overline{G}) \leq 2$, thus by Proposition \ref{prop:radius} we have $\bg(\overline{G}) \leq 3$. Since $G$ is connected and $\rad(G) \leq \floor{\frac{n}{2}}$, Proposition \ref{prop:radius} gives $\bg(G) \leq \frac{n}{2} + 1$. Thus $\bg(G) \bg(\overline{G}) \leq 3 (\frac{n}{2} + 1) = \frac{3n+6}{2} \leq 2n$ for every $n \geq 6$.
    For $4\leq n\leq 7$ a computer check yields $\bg(G)\bg(\overline{G})\leq 2n$. 
\end{proof}

The lower bound in Theorem \ref{thm:product_gen} is tight e.g. when $G = K_{1,n-1}$ and the upper bound is tight e.g. when $G = K_{n}$ (both for $n \geq 2$). When both $G$ and $\overline{G}$ are connected, the upper bound can in fact be improved, as we next show. But first, we recall some definitions and prove some auxiliary results.

Let $k \geq 1$. A set $S \subseteq V(G)$ is a \emph{distance $k$-dominating set} of $G$ if every vertex of $G$ is within distance $k$ from some vertex in $S$. More formally, for every $v \in V(G)$ it holds that $d(v, S) \leq k$. The minimum cardinality of a distance $k$-dominating set of $G$ is the \emph{distance $k$-domination number} $\gamma_k(G)$ of $G$. 

\begin{proposition}
    \label{lem:k-dom}
    If $G$ is a graph, then $\bg(G) \leq \min_{k \geq 1} \{ 2 \gamma_k(G) + k - 1\}$ and $\bg'(G) \leq \min_{k \geq 1} \{ 2 \gamma_k(G) + k\}$.
\end{proposition}

\begin{proof}
    Let $D_k = \{x_1, \ldots, x_\ell\}$ be a distance $k$-dominating set of $G$ where $\ell = \gamma_k(G)$. Burner's strategy is to play vertices $x_1, \ldots, x_\ell$ in his turns (in rounds $1$, \ldots, $2 \ell - 1$). Because every vertex of $G$ is within distance $k$ of some $x_i$, after $k$ additional rounds all vertices of $G$ are burned. Thus $\bg(G) \leq 2 \ell - 1 + k = 2 \gamma_k(G) + k - 1$ and $\bg'(G) \leq 2 \ell + k = 2 \gamma_k(G) + k$.
\end{proof}

\begin{proposition}
    \label{prop:k-dom-connected}
    If $G$ is a connected graph, then $\bg(G) \leq \min_{k \geq 1} \{\gamma_k(G) + 3k\}$ and $\bg'(G) \leq \min_{k \geq 1} \{\gamma_k(G) + 3k + 1\}$.
\end{proposition}

\begin{proof}
    Let $D_k = \{x_1, \ldots, x_\ell\}$ be a distance $k$-dominating set of $G$ where $\ell = \gamma_k(G)$. If $\ell=1$, then Burner plays $x_1$ on his first turn. Since every vertex of $G$ is within distance $k$ of $x_1$, we have $\bg(G)\leq k+1\leq 1+3k$ and $\bg'(G)\leq k+2\leq 1+3k+1$, as desired. We may therefore assume that $\ell\geq 2$. Let $H$ be a graph with vertex set $D_k$ and $x_i x_j \in E(H)$ if and only if $d_G(x_i, x_j) \leq 2k+1$. Since $G$ is connected and all vertices of $G$ are within distance $k$ of some $x_i$, $H$ must also be connected, thus $\gamma(H) \leq \frac{|V(H)|}{2} = \frac{\ell}{2}$. Let $D = \{x_{i_1}, \ldots, x_{i_r}\}$ be a minimum dominating set of $H$, where $r = \gamma(H)$. 

    Burner's strategy is to play vertices from $D$ (in any order). We claim that once Burner has done this, after $3k+1$ additional rounds, all vertices of $G$ are burned.

    Let $v \in V(G)$. Then there exists $x_i \in D_k$ such that $d_G(v, x_i) \leq k$. If $x_i \in D$, then $v$ is burned within the next $k$ rounds after $x_i$ has been played. If $x_i \notin D$, then there exists $x_j \in D$ such that $d_H(x_i, x_j) \le 1$, hence $d_G(x_i, x_j) \le 2k+1$.  Within $2k+1$ rounds after $x_j$ is played, $x_i$ is burned; within another $k$ rounds, $v$ is burned as well.  It follows that $\bg(G) \leq 2 \frac{\ell}{2} - 1 + 3 k + 1 = \ell + 3k = \gamma_k(G) + 3k$. Similarly, $\bg'(G) \leq 2 \frac{\ell}{2} + 3k + 1 = \gamma_k(G) + 3k + 1$.
\end{proof}

By modifying the proof of Theorem \ref{thm:product_gen}, we can obtain a better upper bound if both $G$ and $\overline{G}$ are connected, similarly as was done in~\cite{bonato2016burn}. Since $G$ is connected, Proposition \ref{prop:k-dom-connected} together with the bound $\gamma_k(G) \leq \frac{n}{k+1}$ for all connected graphs on $n \geq k+1$ vertices from \cite{meir1975upper} gives $\bg(G) \leq \frac{n}{k+1} + 3k$ if $n \geq k+1$. Applying this bound for $k=2$ and $n \geq 3$ to the case where $\diam(\overline{G}) \le 2$, we obtain $\bg(G) \bg(\overline{G}) \leq 3 \left( \frac{n}{3} + 6 \right) = n+18$. This gives the following result.

\begin{corollary}
    \label{cor:product-connected-burner}
    If $G$ and $\overline{G}$ are both connected on $n \geq 3$ vertices, then $\bg(G) \bg(\overline{G}) \leq n+18$.
\end{corollary}

We next turn our attention to the Staller-start burning game.

\begin{theorem}
    \label{thm:sum-S-game}
    If $G$ is a graph of order $n \geq 2$, then $4 \leq \bg'(G) + \bg'(\overline{G}) \leq n+2$.
\end{theorem}

\begin{proof}
    As $n \geq 2$, $\bg'(G) + \bg'(\overline{G}) \geq 4$. If $G$ is a complete graph, then $\bg'(G) + \bg'(\overline{G}) = n+2$. If $G$ is not complete but it has a universal vertex $u$, then $\bg'(G) \leq 3$ and $\overline{G}$ contains at least one component of order at least 2. If $n=3$, then $G = P_3$ and $\bg'(G) + \bg'(\overline{G}) = 2 + 3 = n+2$. If $n \geq 4$, then $\bg'(\overline{G}) \leq n-1$ as $\overline{G}$ contains at least one component of order at least 2, so $\bg'(G) + \bg'(\overline{G}) \leq n+2$.

    If $\Delta(G) = n-2$, then let $v$ be a vertex of maximum degree in $G$ and let $w \in V(G) - N[v]$. Clearly, $\bg'(G) \leq 3$. On $\overline{G}$, if Staller starts on $\{v,w\}$, then Burner plays on $N(v)$, thus at most one of $\{v,w\}$ is played during the game (if it lasts at least 3 rounds), while if Staller starts on $N(v)$, then Burner plays $v$ in round 2, again ensuring that at most one of $\{v,w\}$ is played during the game. Thus $\bg'(\overline{G}) \leq \max\{3, n-1\}$. This means that $\bg'(G) + \bg'(\overline{G}) \leq \max\{6, n+2\}$. For $n \geq 4$ the proof in this case is complete. If $n = 3$, then $G$ is a disjoint union of $K_2$ and $K_1$, $\overline{G}$ is $P_3$, so $\bg'(G) + \bg'(\overline{G}) = 3 + 2 = 5 = n + 2$. If $n = 2$, then $G$ is a disjoint union of two $K_1$s and $\overline{G}$ is $K_2$, thus $\bg'(G) + \bg'(\overline{G}) = 2 + 2 = 4 = n + 2$.
    
    Now suppose that $\Delta(G) \leq n-3$. Let $x_1, \ldots, x_k$ be vertices played in an optimal Staller-start burning game on $\overline{G}$. Then $\bg'(\overline{G}) \leq k+1$. As $x_k$ was a legal move, it cannot be adjacent to $x_i$ (in $\overline{G}$) for any $i \in [k-1]$, thus $\Delta(G) \geq \deg_{G}(x_k) \geq k-1$. By Proposition \ref{prop:Delta}, $\bg'(G) \leq n - \Delta(G) \leq n - (k-1)$. Thus $\bg'(G) + \bg'(\overline{G}) \leq n - k + 1 + k+1 = n+2$.
\end{proof}

The upper bound in Theorem \ref{thm:sum-S-game} is tight for example for complete graphs. Note that the lower bound is achieved for $G=K_2$. However, using Proposition \ref{prop:characterize-1-2} and a simple case analysis, one can deduce that if $n \geq 5$, then $\bg'(G) + \bg'(\overline{G}) \geq 6$, which is tight for example for $G = K_{1, n-1}$.

\begin{theorem}
    \label{thm:product-S-game}
    If $G$ is a graph of order $n \geq 6$, then $$ 8 \leq \bg'(G) \bg'(\overline{G}) \leq 3n-6.$$
\end{theorem}

\begin{proof}
    We first prove the lower bound. Without loss of generality, assume that $\bg'(G) \geq \bg'(\overline{G})$. If $\bg'(\overline{G}) = 2$, then by Proposition \ref{prop:characterize-1-2} (4.), $G$ is a disjoint union of $K_2$s and $K_1$s. Thus, since $n \geq 6$, $\bg'(G) \geq 4$, which gives $\bg'(G) \bg'(\overline{G}) \geq 8$. If $\bg'(\overline{G}) \geq 3$, then $\bg'(G) \bg'(\overline{G}) \geq 9$.

    The rest of the proof is devoted to the upper bound.

    \begin{description}
        \item[Case 1.] $G$ is not connected and has at least $n-2$ isolated vertices.\\
        If $G = \overline{K_n}$, then $\bg'(G) = n$ and $\bg'(\overline{G}) = 2$, so $\bg'(G) \bg'(\overline{G}) = 2n \leq 3n-6$ as $n \geq 6$. If $G$ consists of isolated vertices and exactly one copy of $K_2$, then $\bg'(G) = n-1$ and $\bg'(\overline{G}) = 2$ by Proposition \ref{prop:characterize-1-2} (4.), so $\bg'(G) \bg'(\overline{G}) \leq 2(n-1) \leq 3n-6$ as $n \geq 4$.

        \item[Case 2.] $G$ is not connected, has at most $n-3$ isolated vertices and has at least one component of order at most two.\\
        Since $G$ has a component of order at most two, $\diam(\overline{G}) \leq 2$, thus $\bg'(\overline{G}) \leq 3$.
        Since $G$ has at most $n-3$ isolated vertices, it either has a component of order at least 3 or at least two $K_2$s as components.
        \begin{description}
            \item[Case 2a.] $G$ has a component $C$ of order at least 3.\\
            If Burner can play on $C$ first, then he can select a vertex in $C$ with at least two unburned neighbors, thus $\bg'(G) \leq n-2$. If Staller plays on $C$ first, then since Burner was not able to play on $C$ first, there is a different component still available for Burner to play on. Because Burner can avoid playing on $C$ for at least one more round, Staller's move on $C$ causes fire to spread to at least two additional vertices, thus again $\bg'(G) \leq n-2$.
    
            \item[Case 2b.] $G$ has at least two $K_2$s as components.\\
            Whenever an endpoint from $K_2$ is played, the other endpoint is burned in the next round. Since there are at least two $K_2$s as components, we have $\bg'(G) \leq n-2$.
        \end{description}

        In both cases we get $\bg'(G) \bg'(\overline{G}) \leq 3 (n-2) = 3n-6$.

        \item[Case 3.] $G$ is not connected and all of its components are of order at least 3.\\
       As in the proof of Theorem \ref{thm:product_gen} we see that $\diam(\overline{G}) \leq 2$ thus $\bg'(\overline{G}) \leq 3$. Together with Lemma \ref{lem:disconnected-3} this gives $\bg'(G) \bg'(\overline{G}) \leq 3\frac{n+1}{2} \leq 3n-6$ as $n \geq 5$. 

        \item[Case 4.] $G$ and $\overline{G}$ are both connected.\\
        Without loss of generality we may assume that $\bg'(G) \geq \bg'(\overline{G})$. If $\diam(\overline{G}) \geq 3$, then $\diam(G) \leq 3$, thus $\bg'(G) \leq 4$ by Proposition \ref{prop:radius}. By the assumption this means that $\bg'(\overline{G}) \leq 4$ as well, so $\bg'(G) \bg'(\overline{G}) \leq 16 \leq 3n-6$ for every $n \geq 7$. 

        Otherwise, $\diam(\overline{G}) \leq 2$, thus by Proposition \ref{prop:radius} we have $\bg'(\overline{G}) \leq 3$. Since $G$ is connected and $\rad(G) \leq \floor{\frac{n}{2}}$, Proposition \ref{prop:radius} gives $\bg'(G) \leq \frac{n}{2} + 2$. Thus $\bg'(G) \bg'(\overline{G}) \leq 3 (\frac{n}{2} + 2) = \frac{3n}{2} + 6 \leq 3n - 6$ for every $n \geq 8$. 
        
        If $n\in \{6,7,8\}$ and $G, \overline{G}$ are both connected, it can be checked with the help of a computer that the desired bounds hold. \hfill \qedhere
    \end{description}
\end{proof}

It follows from the proof of Theorem \ref{thm:product-S-game} that the lower bound is attained only if $n = 6$ and one of the graphs is a disjoint copy of three $K_2$s. Otherwise, the proof produces a lower bound of 9, which is best possible since for $G = K_{1,n}$ we get $\bg'(G) = 3$ and $\bg'(\overline{G}) = 3$. The upper bound is best possible as shown by $G$ being a disjoint union of $P_3$ and $n-3$ copies of $K_1$ (in this case, $\bg'(G) = n-2$ and $\bg'(\overline{G}) = 3$).

Similarly as before we can obtain a better upper bound if both $G$ and $\overline{G}$ are connected. Combining Proposition \ref{prop:k-dom-connected} with the bound $\gamma_k(G) \leq \frac{n}{k+1}$ for all connected graphs on $n \geq k+1$ vertices from \cite{meir1975upper} gives $\bg'(G) \leq \frac{n}{k+1} + 3k + 1$ if $n \geq k+1$. Applying this bound for $k=2$ and $n \geq 3$ to the case where $\diam(\overline{G}) \le 2$, we obtain $\bg'(G) \bg'(\overline{G}) \leq 3 \left( \frac{n}{3} + 7 \right) = n+21$. This gives the following result.

\begin{corollary}
    \label{cor:product-connected-staller}
    If $G$ and $\overline{G}$ are both connected on $n \geq 3$ vertices, then $\bg'(G) \bg'(\overline{G}) \leq n+21$.
\end{corollary}

\section{Graph products}
\label{sec:products}

In this section we explore the burning game on graph products. Recall that general bounds for the burning number of Cartesian, strong and lexicographic products were explored in 
\cite{mitsche2018products}, Cartesian and strong grids were studied first in \cite{mitsche2017probabilistic} and later also in \cite{bonato2021fence}, and the burning of hypercubes was actually resolved under a different name already in 1992 by Alon \cite{alon1992transmitting}. The cooling number of square grids was given in \cite{bonato2024cool}. We start by recalling the definitions. 

The \emph{Cartesian product} $G \cart H$ of graphs $G=(V(G),E(G))$ and $H=(V(H),$ $E(H))$
has vertex set $V(G)\times V(H),$ and vertices $(u,v),(x,y)$ are adjacent whenever $u=x$ and $vy\in
E(H)$, or $ux\in E(G)$ and $v=y$. 

The \emph{strong product}  $G \strongprod H$ of graphs $G=(V(G),E(G))$ and $H=(V(H),$ $E(H))$
has vertex set $V(G)\times V(H),$ and vertices $(u,v),(x,y)$ are adjacent whenever $u=x$ and $vy\in
E(H)$, or $ux\in E(G)$ and $v=y$, or $ux \in E(G)$ and $vy \in E(H)$.

The \emph{lexicographic product}  $G[H]$ of graphs $G=(V(G),E(G))$ and $H=(V(H),$ $E(H))$
has vertex set $V(G)\times V(H),$ and vertices $(u,v),(x,y)$ are adjacent whenever $ux \in E(G)$, or
$u=x$ and $vy \in E(H)$.

Let $G$ and $H$ be arbitrary graphs, and $v \in V(H)$. We refer to the set $V(G) \times \{v\}$ as a \emph{$G$-layer}. Similarly, the set $\{u\} \times V(H)$ for $u \in V(G)$ is an \emph{$H$-layer}. When referring to a specific $G$- or $H$-layer, we denote them by $G^v$ or $^uH$, respectively. Layers can also be regarded as the graphs induced on these sets. Obviously, in the Cartesian, strong and lexicographic products, a $G$-layer or $H$-layer is isomorphic to $G$ or $H$, respectively.

The \emph{corona product} of two graphs $G$ and $H$, denoted by $G \circ H$, is defined as the graph obtained by taking one copy of $G$ and $|V(G)|$ copies of $H$ and joining the $i$-th vertex of $G$ to every vertex in the $i$-th copy of $H$.

We begin with a general result that is particularly useful for analyzing products of graphs. Recall that a \textit{homomorphism} from a graph $G$ to a graph $H$ is a map $\varphi : V(G) \mapsto V(H)$ such that whenever $uv \in E(G)$, we have $\varphi(u)\varphi(v) \in E(H)$.  If in fact $H$ is a subgraph of $G$ (and $\varphi$ restricted to $V(H)$ is the identity map), then we call $\varphi$ a \textit{retraction} and say that $H$ is a \textit{retract} of $G$.

In general, the game burning number of a graph $G$ can be much smaller than the game burning number of a subgraph $H$, since $G$ may contain some additional structure that facilitates burning the graph.  However, this cannot happen if $H$ is a retract of $G$, as we next show.  Recall that for a graph $G$ and $S \subseteq V(G)$, we denote by $G \vert S$ the graph $G$ with the vertices in $S$ regarded as already burning.

\begin{theorem}\label{thm:retract} 
For graphs $G$ and $H$, if $H$ is a retract of $G$ with retraction $\varphi$, then for all $S \subseteq V(G)$ we have $\bg(H \vert \varphi(S)) \le \bg(G \vert S)$ and $\bg'(H \vert \varphi(S)) \le \bg'(G \vert S)$.  In particular, $\bg(H) \le \bg(G)$ and $\bg'(H ) \le \bg'(G)$.
\end{theorem}
\begin{proof}
We prove that $\bg(H \vert \varphi(S)) \le \bg(G \vert S)$; the proof that $\bg'(H \vert \varphi(S)) \le \bg'(G \vert S)$ is similar.  To this end, we give a Burner strategy for the burning game on $H$ that will let Burner ensure that the game on $H$ lasts at most $\bg(G \vert S)$ rounds.  Loosely, Burner will use an optimal strategy on $G$ to guide his play on $H$.

Burner will play two games simultaneously: the ``real'' game on $H \vert \varphi(S)$ and an imagined game on $G \vert S$.  Burner plays on $H \vert \varphi(S)$ as follows.  Let $\varphi$ be a retraction from $G$ to $H$.  On each Burner turn in the game on $H \vert \varphi(S)$, Burner chooses an optimal move for the imagined game on $G \vert S$, say $v$; Burner then plays $v$ on $G$ and plays $\varphi(v)$ on $H$.  (If $\varphi(v)$ is already burned in the game on $H$, then to simplify the analysis, we suppose that Burner ``passes'' and burns nothing; by the Continuation Principle, this cannot shorten the length of the game, so the upper bound we obtain for $\bg(H \vert \varphi(S))$ will still be valid.)  Suppose instead that it is Staller's turn and that Staller plays vertex $w$ on $H$; Burner imagines that Staller likewise plays $w$ in the game on $G$.  (If in fact $w$ has already been burned on $H$, then Burner imagines that Staller plays any valid move; however, we will argue later that this scenario can never actually arise.)

At each point during the game, let $S_G$ and $S_H$ denote the sets of burned vertices in the game on $G \vert S$ and the game on $H \vert \varphi(S)$, respectively.  We claim that always $S_H \supseteq \varphi(S_G)$.  By surjectivity of $\varphi$, it will then follow that by the time $V(G)$ has completely burned, so has $V(H)$, and so the game on $H \vert \varphi(S)$ lasts no longer than the game on $G \vert S$.  Since Burner follows an optimal strategy in the game on $G \vert S$, the length of the game on $G \vert S$ -- and hence also the length of the game on $H \vert \varphi(S)$ -- is at most $\bg(G \vert S)$, i.e. $\bg(H \vert \varphi(S)) \le \bg(G \vert S)$.

At the outset of the game, we have $S_G = S$ and $S_H = \varphi(S)$, hence $S_H \supseteq \varphi(S_G)$ as needed.  To show that in fact $S_H \supseteq \varphi(S_G)$ at all times, we must argue that this property is maintained when Burner selects vertices, when Staller selects vertices, and when fire spreads.  

\begin{itemize}
\item First, consider the selection of vertices by Burner.  Suppose that during some round of the game, Burner selects vertex $v$ in $G$ and thus selects vertex $\varphi(v)$ in $H$ (or no vertex at all, if $\varphi(v)$ is already burning.)  The set of burned vertices after Burner's move is $S_H \cup \{\varphi(v)\}$, and clearly $S_H \cup \{\varphi(v)\} \supseteq \varphi(S_G) \cup \{\varphi(v)\} = \varphi(S_G \cup \{v\})$, so the property is maintained.

\item Next, consider the selection of vertices by Staller.  Because $S_H \supseteq \varphi(S_G)$, and because $\varphi$ is the identity on $V(H)$, whatever vertex Staller chooses on $H$ must not yet be burned in $G$; hence Staller plays the same vertex in $G$ as in $H$, which clearly preserves the property that $S_H \supseteq \varphi(S_G)$.

\item Finally, we must argue that the property is maintained in the course of a spreading phase.  Let $S_G$ and $S'_G$ denote the sets of burned vertices in $G$ before and after the spreading phase, respectively; likewise, let $S_H$ and $S'_H$ denote the sets of burned vertices in $H$ before and after the spreading phase.  Assume that $S_H \supseteq \varphi(S_G)$, i.e. that the desired property holds before fire spreads.  
Consider a vertex $x$ in $\varphi(S'_G)$; we will argue that $x \in S'_H$.  
If $x \in \varphi(S_G)$, then $x \in \varphi(S_G) \subseteq S_H \subseteq S'_H$.  Suppose instead that $x \in \varphi(S'_G) \setminus \varphi(S_G)$.  It follows that $x$ has some preimage $v$ such that $v \in S'_G \setminus S_G$.  Consequently, it must be that $v$ has some neighbor $w$ (in $G$) such that $w \in S_G$.  Since $v$ and $w$ are adjacent in $G$, and since $\varphi$ is a homomorphism, $\varphi(v)$ and $\varphi(w)$ must be adjacent in $H$.  Since $w \in S_G$, we have $\varphi(w) \in \varphi(S_G) \subseteq S_H$, i.e. $\varphi(w)$ is already burning in $H$ before the fire spreads; the fire will thus spread to $\varphi(v)$.  Hence $x = \varphi(v) \in S'_H$, as desired.
\end{itemize}

Thus at all times we have $S_H \supseteq \varphi(S_G)$, from which it follows that $\bg(H \vert \varphi(S)) \le \bg(G \vert S)$.
\end{proof}




\begin{proposition}
\label{cartesian-product}
Let $G$ and $H$ be connected graphs. Then
$$\max\{\bg(G),\bg(H)\}  \leq \bg(G \strongprod H) \leq \bg(G \cart H) \leq \min\{2b(G^2) - 1 + \rad(H), 2 b(H^2) - 1 + \rad(G)\}$$
and
$$\max\{\bg'(G),\bg'(H)\}  \leq \bg'(G \strongprod H) \leq \bg'(G \cart H) \leq \min\{2b(G^2) + \rad(H), 2 b(H^2) + \rad(G)\}.$$
\end{proposition}
\begin{proof}
Since both $G$ and $H$ are retracts of $G \boxtimes H$, Theorem \ref{thm:retract} yields
$$b_g(G \boxtimes H) \ge \max\{b_g(G), b_g(H)\}.$$
The inequality $\bg(G \boxtimes H) \leq \bg(G \, \Box \, H)$ follows from Lemma~\ref{lem:spanning-subgraph} as the Cartesian product $G \, \Box \, H$ is a spanning subgraph of $G \boxtimes H$. 

To prove the upper bound, let $b(G^2) = k$, $\rad(H) = r$ and let $u \in V(H)$ be a central vertex. Let $x_1, \dots, x_k$ be an optimal sequence of sources for the burning process on $G^2$. Burner's strategy is to play $(x_1, u), \ldots, (x_k,u)$ in his first $k$ moves, and to play any legal move in the remaining rounds.  (If any $(x_i,u)$ burns before Burner plays it, then he plays any legal move instead.) It follows from the arguments used in the proof of Proposition \ref{prop:trivial-bounds} that after $2k-1$ rounds of the game, all vertices in $G^u$ are burned. Thus in the remaining $r$ rounds, all vertices of $G \cart H$ are burned. Thus $\bg(G \cart H) \leq 2k-1+r$. By symmetry we also have $\bg(G \cart H) \leq 2 b(H^2) - 1 + \rad(G)$.

The proofs of the bounds on $\bg'(G \strongprod H)$ and $\bg'(G \cart H)$ are similar. (The difference in the upper bound is due to the difference in the analogous bound in Proposition \ref{prop:trivial-bounds}.)
\end{proof}

Note that the above chain of bounds is very similar to the bounds known for the burning number. It was proven in \cite{mitsche2018products} that $\max\{b(G),b(H)\}  \leq b(G \strongprod H) \leq b(G \cart H) \leq \min \{ b(G) + \rad(H), b(H) + \rad(G) \}$.

\begin{example}
    Let $G$ and $H$ be graphs on at least two vertices, and with universal vertices $u_1$ and $u_2$, respectively. Proposition \ref{cartesian-product} gives $2 \leq \bg(G \strongprod H) \leq 4$. But if Burner starts the game by playing $(u_1, u_2)$, all vertices are burned by the end of the spreading phase of round 2, thus the lower bound is attained. 

\end{example}

\begin{example}
    Let $n \geq 3$ and consider $\bg(K_n \cart P_4)$. Proposition \ref{cartesian-product} gives $2 \leq \bg(K_n \cart P_4) \leq 4$. We will show that the upper bound is attained. Let $V(K_n) = [n]$ and $V(P_4) = [4]$. Without loss of generality Burner starts the game by playing a vertex $(1, i)$ where $i \in [2]$. If Staller responds by playing $(2, 2-i+1)$, vertices $(2,4), \ldots, (n,4)$ are not burned by the end of the spreading phase of round 3, thus as $n \geq 3$, at least 4 rounds are needed.
\end{example}

We next consider the burning game played on the $n$-dimensional hypercube $Q_n$, i.e. the $n$-fold Cartesian product $K_2 \cart K_2 \cart \dots \cart K_2$.  Alon~\cite{alon1992transmitting} proved that in the original burning process, burning $Q_n$ requires at least $\ceil{\frac{n}{2}}+1$ rounds; the burning game on $Q_n$ behaves quite similarly to this.  (In fact, Alon argued that for the burning process, if $n$ is even, then at least $\frac{n}{2}+1$ rounds are needed even if a pair of antipodal vertices may be burned in every round.)  

\begin{theorem}\label{thm:hypercubes}
For $n\geq 1$, 
\[\bg(Q_n) = \left\{\begin{array}{ll}2, &\text{ if } n \in \{1, 2\};\\\ceil{\frac{n+1}{2}}+1, \quad &\text{ otherwise}\end{array}\right . \quad \text{ and } \quad \bg'(Q_n) = \ceil{\frac{n}{2}}+1.\]
\end{theorem}
\begin{proof}
It is easily seen that $\bg(Q_1) = \bg(Q_2) = 2$, so suppose $n \ge 3$.  We will view the vertices of $Q_n$ as ordered $n$-tuples with each coordinate in $\{0,1\}$, where two vertices are adjacent if and only if the corresponding $n$-tuples differ in exactly one coordinate.

We begin with the upper bounds.  Let $r = \ceil{\frac{n+1}{2}}+1$.  For the upper bound on $\bg(Q_n)$, we give a strategy for Burner to burn $Q_n$ in $r$ rounds.  Burner burns vertex $(0,\dots,0)$ in round 1 and vertex $(1,\dots,1)$ in round 3. (If vertex $(1,\dots,1)$ is already burned before Burner's second move, then he instead plays any legal move.) 
For the remainder of the game, Burner plays arbitrarily.  We claim that by the end of round $r$, all vertices of $Q_n$ will be burned.  Note that vertex $(0,\dots,0)$ was burned in round 1, so by the end of round $r$, all vertices within distance $r-1$ of $(0,\dots,0)$ will be burned.  Likewise, since vertex $(1,\dots,1)$ was burned on or before round 3, all vertices within distance $r-3$ of $(1,\dots,1)$ will be burned.  Consider an arbitrary vertex $v$ in $Q_n$, and let $k$ denote the number of coordinates of $v$ with value 1.  If $k \le r-1$ then $v$ has been burned due to its proximity to $(0,\dots,0)$.  Otherwise, we have $2r = 2(\ceil{(n+1)/2}+1) \ge n+3$; hence $k \ge r \ge n-(r-3)$, so the distance from $v$ to $(1,\dots,1)$ is at most $r-3$, and thus $v$ has been burned due to its proximity to $(1,\dots,1)$. 

A similar argument suffices to establish the upper bound on $\bg'(Q_n)$.  Let $r = \ceil{\frac{n}{2}}+1$.  By symmetry, we may assume that Staller burns vertex $(0,\dots,0)$ in round 1.  In round 2, Burner burns $(1,\dots,1)$.  For the remainder of the game, Burner plays arbitrarily.  By the end of round $r$, all vertices within distance $r-1$ of $(0,\dots,0)$ have been burned, as have all vertices within distance $r-2$ of $(1,\dots,1)$.  Given a vertex $v$ with $k$ coordinates equal to 1, either $k \le r-1$ or $k \ge r \ge n-r+2$; thus either $v$ is within distance $r-1$ of $(0,\dots,0)$ or it is within distance $r-2$ of $(1,\dots,1)$.  In either case, $v$ will be burned by the end of round $r$.

We next consider the lower bounds.  For the Staller-start game, Proposition~\ref{prop:trivial-bounds} and Alon's result in \cite{alon1992transmitting} together yield 
\[\bg'(Q_n) \ge b(Q_n) = \ceil{\frac{n}{2}}+1,\]
as desired.  Similarly, for the Burner-start game, we have 
\[\bg(Q_n) \ge b(Q_n) = \ceil{\frac{n}{2}}+1;\]
when $n$ is odd, this establishes the desired lower bound on $\bg(Q_n)$ (since in this case $\ceil{\frac{n+1}{2}} = \ceil{\frac{n}{2}}$).  Finally, suppose that $n$ is even, and let $n=2k$; we give a strategy for Staller to ensure that the graph cannot be fully burned before round $k+2$.

Without loss of generality, we may suppose that Burner plays vertex $v_1 = (0,0,0,\dots,0)$ in round 1.  In round 2, Staller plays $v_2 = (1,1,0,\dots,0)$.  If $n=4$ then it is easily seen that the graph cannot be fully burned by the end of round 3, 
so suppose $n \ge 6$.  Let $v_3$ denote the vertex played by Burner in round 3.  In round 4, Staller plays any vertex in which exactly two of the first three coordinates are 1s, as are exactly two of the last $n-3$.  The number of such vertices in $Q_n$ is $3\binom{n-3}{2}$.  To see that some such vertex has not yet been burned, note that there are no such vertices within distance 3 of $v_1$, exactly $\binom{n-3}{2}$ within distance 2 of $v_2$, and at most $n-3$ within distance 1 of $v_3$; hence at least one vertex of this form remains unburned provided that 
\[3\binom{n-3}{2} > \binom{n-3}{2}+n-3,\]   
which holds whenever $n \ge 6$.  For simplicity, suppose that $v_4 = (1, 0, 1, 1, 1, 0, \dots, 0)$; the other cases are similar.  

Define a subgraph $H$ of $Q_n$ as follows: if the first three coordinates of $v_3$ are not all 1, then $H$ consists of all vertices of $Q_n$ in which the first three coordinates are 1; otherwise, $H$ consists of all vertices of $Q_n$ in which the first three coordinates are 0, 1, and 1, in that order.  In either case, note that $H$ is isomorphic to $Q_{n-3}$.  We claim that at the end of round $k+1$, at least one vertex of $H$ will remain unburned.  To see this, let us first take stock of the status of the game at the end of round 4.  All vertices within distance 3 of $v_1$ have burned, as have all vertices within distance 2 of $v_2$, all vertices within distance 1 of $v_3$, and vertex $v_4$.  Now let us restrict our attention to the burned vertices in $H$.  By choice of $H$, no $v_i$ belongs to $H$.  The only burned vertices of $H$ within distance 3 of $v_1$ or distance 2 of $v_2$ are the ``all-zeroes'' vertex of $H$ and (possibly) its neighbors; since $v_3$ does not agree with the vertices of $H$ in the first three coordinates, at most one vertex of $H$ is within distance 1 of $v_3$; and $v_4$ is not in $H$.  In any case, focusing exclusively on $H$, letting $y$ be the ``all-zeroes'' vertex of $H$ and letting $z$ be the unique vertex of $H$ closest to $v_3$, all burned vertices of $H$ belong to $N_H[y] \cup \{z\}$.  

Moreover, note that $y$ is the unique vertex of $H$ closest to both $v_1$ and $v_2$, while $z$ is the unique vertex of $H$ closest to $v_3$, and the unique vertex of $H$ closest to $v_4$ is either $(1, 1, 1, 1, 1, 0, \dots, 0)$ or $(0, 1, 1, 1, 1, 0, \dots, 0)$, depending on how $H$ was chosen.

We claim that it will take at least another $\bg(H \vert (N_H[y] \cup \{z\}))$ rounds to burn all vertices of $H$.  As noted above, after the first four rounds of the game, all burned vertices of $H$ belong to $N_H[y] \cup \{z\}$.  Thus, the claim will follow by the Continuation Principle (Theorem \ref{thm:continuation-principle}), provided we can show that the vertices of $G-H$ have no impact on the number of additional rounds needed to burn all of $H$. 
To see this, first note that for any vertex $v$ of $H$, some shortest path to $v$ from $v_1$ or $v_2$ must pass through $y$, while some shortest path to $v$ from $v_3$ must pass through $z$; consequently, fire that spreads to $v$ from $v_1, v_2$, or $v_3$ can be viewed as spreading through $y$ or $z$.  Fire that spreads to $v$ from $v_4$ can be viewed as spreading through either $(1, 1, 1, 1, 1, 0, \dots, 0)$ or $(0, 1, 1, 1, 1, 0, \dots, 0)$ -- depending on how $H$ was chosen -- but this vertex will burn due to proximity to $v_2$ at the same time it will burn due to proximity from $v_4$, so the fire at $v_2$ does not cause any additional vertices in $H$ to burn beyond those that would already burn on account of $v_4$.  As such, we may safely ``ignore'' all vertices outside of $H$ that were burned within the first four rounds of the game.  Similarly, to minimize the number of rounds needed to burn all of $H$, there is no advantage to playing outside of $H$: instead of playing some vertex $w$ in $G-H$, it would be at least as effective to play the unique vertex of $H$ closest to $w$.  It now follows that even if the players focus solely on $H$, while ignoring the vertices of $G-H$, it will still take at least $\bg(H \vert (N_H[y] \cup \{z\}))$ more rounds to accomplish the task.

Since $H$ is isomorphic to $Q_{n-3}$, the game on $H$ relative to $N_H[y] \cup \{z\}$ can be viewed as a game-in-progress on $Q_{n-3}$, in which two moves have already been played -- the first on $y$ and the second on $z$.  By Alon's result, any burning process on $Q_{n-3}$ must last at least $\ceil{\frac{n-3}{2}}+1$ rounds; thus
\[\bg(H \vert (N_H[y] \cup \{z\})) \ge b(Q_{n-3}) - 2 \ge \ceil{\frac{n-3}{2}}-1 = \ceil{\frac{2k-3}{2}}-1 = k-2,\]
so at least another $k-2$ rounds are needed to burn the rest of $H$.  Since four rounds have already elapsed in the game on $G$, the total length of the game must be at least $k-2+4$, i.e. $k+2$, as claimed.
\end{proof}

We next turn our attention to corona products. We establish bounds for both versions of the burning game in terms of the burning number of the square of the first factor.

\begin{theorem}
\label{corona-product}
Let $G$ and $H$ be connected graphs. Then
$$2b(G^2)-2 \leq \bg(G \circ H) \leq 2b(G^2),$$
where equality in the lower bound requires $H = K_1$.  Additionally,
$$2b(G^2)-3 \leq \bg'(G \circ H) \leq 2b(G^2)+1,$$
where once again, if $H \not = K_1$, we cannot have equality in the lower bound.
\end{theorem}
\begin{proof}
We prove the bounds on $\bg(G \circ H)$; the bounds on $\bg'(G \circ H)$ then follow using \cite[Proposition 5]{burning-1}, which states that for any graph, $\bg$ and $\bg'$ differ by at most 1.

Consider the Burner-start game on $G \circ H$.  For the upper bound, consider an optimal burning sequence $v_1, v_2, \dots, v_k$ for $G^2$. 
Burner plays $v_1, v_2, \dots, v_k$ in order on his first $k$ turns.  (If any of these vertices burns before Burner plays it, he plays any legal move instead.)  By the arguments used in the proof of Proposition \ref{prop:trivial-bounds}, Burner's moves ensure that all of $G$ will burn within $2b(G^2)-1$ rounds. Because of the structure of $G\circ H$, within one more round, all copies of $H$ will burn as well. In the Staller-start game, the bound is larger by 1 due to the possibility that Staller may play an unhelpful first move.

For the lower bound, we give a strategy for Staller.  First, note that we may assume that Burner plays in $G$ whenever possible.  Given any vertex $v$ in $G$ and any vertex $w$ in the corresponding copy of $H$, we have $N_{G \circ H}[v] \supset N_{G \circ H}[w]$, so by the Continuation Principle, playing $v$ is always at least as good for Burner as playing $w$.  If Burner is considering playing a vertex $w$ in some copy of $H$ for which the corresponding vertex $v$ of $G$ has already burned, then playing any unburned vertex of $G$ will be at least as good, since no matter what, all vertices in the copy of $H$ will burn during the next spreading phase on account of being adjacent to $v$.  Thus, we may assume that Burner never plays in a copy of $H$ unless there is no other choice. 

Now, we explain Staller's strategy.  In her first turn, Staller plays a vertex from any copy of $H$; in all subsequent turns, she plays any unburned vertex from a copy of $H$ where the corresponding vertex of $G$ has already burned.  If all vertices of $G$ have already burned, then this is clearly possible.  If instead not all vertices of $G$ have burned, then since $G$ is connected, fire must have spread to at least one new vertex $v$ of $G$ in the last spreading phase; the corresponding copy of $H$ cannot yet have any burned vertices, since otherwise $v$ would already be burning.  

Using this strategy, Staller only plays in copies of $H$ for which the corresponding vertex of $G$ has already burned.  Hence, she does not contribute to the burning of $G$. In the Burner-start game, arguments similar to those used to prove the upper bound in Proposition \ref{prop:trivial-bounds} show that, under an optimal strategy, Burner's moves must correspond to a burning sequence for $b(G^2)$.  Hence, after Burner's $(b(G^2)-1)$th turn -- which happens in round $2b(G^2)-3$ of the Burner-start game on $G \circ H$ -- not all of $G$ can yet be burning.  Thus it takes at least $2b(G^2)-2$ rounds to burn all of $G$, and hence at least this many rounds to burn $G \circ H$.  Note that if $H$ has at least two vertices, then there will be at least two unburned vertices after the spreading phase of round $2b(G^2)-2$, so in fact the game will last for at least $2b(G^2)-1$ rounds.
\end{proof}

\medskip

We next discuss the sharpness of the bounds in
Theorem~\ref{corona-product}. The following examples show that all of the bounds, with the possible exception of the lower bound for the Staller-start game, are sharp.

\begin{example}\label{corona-tightness}
We first consider the Burner-start game on $G \circ H$.

For the lower bound of $2b(G^2)-2$, let $G = K_2$ and $H = K_1$; it is easily checked that $b_g(G \circ H) = b_g(P_4) = 2$, while $b(G^2) = 2$. As noted in Theorem~\ref{corona-product}, if $H \not = K_1$, then we have an improved lower bound of $2b(G^2)-1$.  To see that this bound is tight, let $G$ be any graph on at least three vertices with a universal vertex $v$, let $H$ be any connected graph, and consider the Burner-start burning game on $G\circ H$.  Note that in this case $b(G^2)=2$. Burner can start the game by burning $v$, which will ensure that fire spreads to all vertices of $G$ in the spreading phase of round $2$.  Since $G$ has at least three vertices, at least two copies of $H$ will be completely unburned prior to Staller's move in round 2.  No matter how Staller plays, the rest of the graph will burn during the spreading phase of round 3.

For the upper bound, let $G$ be the subdivided star with central vertex $v$ and at least two leaves, and $H\cong K_1$. Note that $b(G^2)=2$, since $v$ is universal in $G^2$. Burner will start the game by burning $v$.  After the spreading phase of round 2, all subdivision vertices in $G$ are burning, so by the Continuation Principle, it is optimal for Staller to play in any copy of $H$ joined to a subdivision vertex (since all such vertices would burn in the next round regardless of how Staller plays).  $G$ will finish burning during the spreading phase of round $3$.  Since $G$ has at least two leaves, at least two copies of $H$ remain completely unburned and Burner cannot burn both in the selection phase of round $3$, so the game will last at least until round 4 (and, indeed, will necessarily end after the spreading phase of round 4).

\medskip

Now consider the Staller-start burning game on $G\circ H$.

For the lower bound when $H \not = K_1$, let $G = C_{17}$ and let $H$ be any connected graph.  It is easily checked that $b(G^2) = 4$. For the game on $G \circ H$, no matter what Staller's first move is, at least one vertex of $G$ will burn during the spreading phase of round 2.  It is straightforward to check that Burner can ensure that the last vertices of $G$ burn during the spreading phase of round 5, so that the remainder of the graph burns no later than the spreading phase of round 6.


For the upper bound, let $G$ be $P_5$ on vertex set $\{v_1, \dots, v_5\}$ with an extra pendant vertex $v'_2$ attached to $v_2$, and let $H = K_2$. Note that $b(G^2)=2$, since $v_3$ is universal in $G^2$. Staller will start the game by burning a vertex in the copy of $H$ adjacent to $v'_2$. In the spreading phase of round $2$ the fire will spread only to the other vertex of that copy of $H$ and to vertex $v'_2$, and no matter what Burner chooses to burn in round $2$, after the spreading phase of round $3$, either $v_1$ or $v_5$ and their entire copy of $H$ will still be unburned.  By the Continuation Principle, it is optimal for Staller to play anywhere in a copy of $H$ adjacent to $v_2$ or $v_4$, since all such vertices will burn during the next spreading phase anyway. In the spreading phase of round $4$, fire spreads to the last unburned vertex of $G$, but the attached copy of $H$ will still be unburned. Since $H$ has two vertices, Burner can't burn both, and we need a fifth round to fully burn the graph, thus attaining the upper bound.
\end{example}

\medskip

Thus, the examples above settle the sharpness of all bounds in
Theorem~\ref{corona-product}, except for the lower bound in the Staller-start game. For this bound, Theorem~\ref{corona-product} gives $b'_g(G\circ H)\ge 2b(G^2)-3$, with equality possible only when $H=K_1$. We do not know any example attaining this bound and, in fact, we believe that it is never attained. We can prove this for the first few values of $b(G^2)$.

\begin{proposition}\label{GcircK1-lower}
Let $G$ be a connected graph such that $b(G^2)\le 3$. Then
$$b'_g(G\circ K_1)\ge 2b(G^2)-2.$$
\end{proposition}

\begin{proof}
If $b(G^2)=1$, then the result is trivial, since $2b(G^2)-2=0$. Suppose now that $b(G^2)=2$. Since $b'_g(G\circ K_1)\ge 2$, we immediately obtain $b'_g(G\circ K_1)\ge 2=2b(G^2)-2$.

It remains to consider the case $b(G^2)=3$. For every $x\in V(G)$, let $x'$ denote the pendant vertex adjacent to $x$ in $G\circ K_1$. Staller starts by playing an arbitrary pendant vertex $v'$. Hence, in the spreading phase of round $2$, the vertex $v$ becomes burned. We may assume that Burner then plays a vertex $u\in V(G)$, $u \neq v$, since playing in $G$ is at least as effective as playing a pendant vertex.

Suppose, to the contrary, that the game ends in round $3$. Before the spreading phase of round $3$, the only burned vertices of $G$ are $v$ and $u$. Hence, during the spreading phase of round $3$, the fire may spread from $v$ and $u$ to other vertices of $G$, possibly even burning all remaining vertices of $G$. However, if $x\in V(G)\setminus\{u,v\}$ becomes burned in this spreading phase, then its pendant vertex $x'$ cannot become burned in the same spreading phase. Therefore, after the spreading phase of round $3$, every pendant vertex $x'$ with $x\in V(G)\setminus\{u,v\}$ is still unburned. Since $b(G^2)=3$ and $G$ is connected, we have $|V(G)|\ge 4$; otherwise, $G^2$ would be complete and hence $b(G^2)\le 2$. Thus, at least $|V(G)|-2\ge 2$ pendant vertices are still unburned after the spreading phase of round $3$. Staller can select at most one of these vertices in the selection phase of round $3$. Hence, at least one vertex remains unburned, and the game cannot end in round $3$. Therefore,
$b'_g(G\circ K_1)\ge 4=2b(G^2)-2$.
\end{proof}

Proposition~\ref{GcircK1-lower} shows that the lower bound for the
Staller-start game can be improved for small values of $b(G^2)$. Together with the fact that we do not know any example attaining the lower bound from Theorem~\ref{corona-product}, this suggests that the stronger bound may hold in general. Since the case $H\neq K_1$ already follows from Theorem~\ref{corona-product}, it remains only to consider $H=K_1$. We therefore propose the following conjecture.

\begin{conjecture}
\label{corona-conjecture}
For all connected graphs $G$ and $H$,
$$b'_g(G\circ H)\ge 2b(G^2)-2.$$
\end{conjecture}

\medskip

Lastly, we consider the lexicographic product. We establish bounds for both versions of the burning game, distinguishing between the cases when the second factor has a universal vertex and when it does not.

\begin{theorem}
    \label{lexicographic-product}
    Let $G$ and $H$ be connected graphs with $|V(G)|\geq 2$. If $H$ has a universal vertex, then $$\bg(G) \leq \bg(G[H]) \leq \bg(G) + 1$$
    and $$\bg'(G) \leq \bg'(G[H]) \leq \bg'(G) + 1.$$
    

    If $H$ does not have a universal vertex, then $$2 b(G^2) - 2 \leq \bg(G[H]) \leq 2 b(G^2)$$
    and $$2 b(G^2) - 3 \leq \bg'(G[H]) \leq 2 b(G^2) + 1.$$

\end{theorem}

\begin{proof}
First note that fire cannot spread from a vertex $(u,v)$ to $(u',v')$ unless either $uu'\in E(G)$ or $u=u'$, just as in the game on $G$. Consider the projection of the game on $G[H]$ onto $G$, where every selected vertex $(u,v)$ is projected to $u$. Whenever fire spreads between two different copies of $H$ in $G[H]$, the corresponding vertices of $G$ are adjacent. Hence the projected game cannot burn $G$ faster than the corresponding burning game on $G$. Therefore $\bg(G)\leq \bg(G[H])$ and $\bg'(G)\leq \bg'(G[H])$. This establishes the lower bounds in the case where $H$ has a universal vertex.
    
    We next consider the upper bounds, again in the case where $H$ has a universal vertex.  Let $z$ denote the universal vertex in $H$. 
    Burner imagines a game on $G$ and uses an optimal strategy in that game to guide his play on $G[H]$.  On each of his turns, Burner chooses and plays an optimal move $u$ in the game on $G$; in the game on $G[H]$, he then plays the vertex $(u,z)$.  At any given point, let $S$ denote the set of vertices that are burned in the game on $G$, and let $S'$ denote the set of vertices $u$ in $G$ such that in the game on $G[H]$, all vertices in the copy of $H$ corresponding to $u$ have burned; we claim that throughout the game, immediately before the selection phase of each of Staller's turns, we have $S = S'$.  Note that when fire spreads from a vertex $(u,v)$, for every $u' \in N_G(u)$, every vertex in the copy of $H$ corresponding to $u'$ burns.  Moreover, if one player plays a vertex $(u,v)$, then for any $u' \in N_G(u)$, within two turns fire can spread from the copy of $H$ corresponding to $u$ to the copy of $H$ corresponding to $u'$ and then back.  Finally, when Burner burns a vertex $(u,z)$, during the subsequent spreading phase fire spreads to all vertices in the copy of $H$ corresponding to $u$, due to $z$ being a universal vertex in $H$.  Together, all of this implies that immediately prior to any Staller turn, we have $S = S'$.  Thus when Staller burns a vertex $(u',v')$ in the game on $G[H]$, we must have $u' \not \in S'$ and hence $u' \not \in S$, so Burner may imagine that Staller burned $u'$ in the game on $G$.  
    
    It follows that within $\bg(G)$ rounds, the game on $G$ ends; we claim that the game on $G[H]$ can last at most one round longer.  If the last round of the game on $G$ was Burner's turn, then after the spreading phase of the next round -- i.e. immediately prior to the selection phase of Staller's turn -- we will have $S = S' = V(G)$, so the game on $G[H]$ will also end.  Suppose instead that the last round of the game on $G$ was Staller's turn.  If the game ended during the spreading phase, then again $S = S' = V(G)$ and the game on $G[H]$ is over.  Otherwise, let $x$ denote the last vertex played in the game on $G$.  Since the game on $G$ ended during a selection phase, it must be that all neighbors of $x$ are already burning.  Hence, in the game on $G[H]$, during the next spreading phase, fire will spread to the copy of $H$ corresponding to $x$ from the copy of $H$ corresponding to some neighbor of $x$; again the game on $G[H]$ ends at most one turn later than the game on $G$.  In any case, because Burner follows an optimal strategy for the game on $G$, the game on $G[H]$ lasts at most $\bg(G)+1$ rounds.  Thus $\bg(G[H]) \le \bg(G)+1$; a very similar argument shows that $\bg'(G[H]) \le \bg'(G)+1$.


Now suppose that $H$ does not have a universal vertex. We give a strategy for Staller. Suppose that Burner plays a vertex $(u,v)$. If, after the spreading phase of the subsequent round, there is still an unburned vertex $(u,v')$ in the copy of $H$ corresponding to $u$, then Staller plays any such vertex.

Suppose instead that the copy of $H$ corresponding to $u$ is completely burned after this spreading phase. Then Burner's move does not contribute a new source to the spreading of fire between the copies of $H$. Indeed, if this copy contained no burning vertex before Burner's move and no neighboring copy contained a burning vertex immediately before the subsequent spreading phase, then during this spreading phase fire could spread inside the copy corresponding to $u$ only from $(u,v)$. Since $H$ has no universal vertex, some vertex $(u,v')$ would remain unburned, a contradiction. Thus, in this case Staller plays any legal move.

Consequently, in every pair consisting of a Burner turn and the subsequent Staller turn, at most one of the two moves contributes a new source to the spreading of fire between the copies of $H$. If this is Staller's move, we may regard it as having been played at the beginning of the pair of rounds, which can only accelerate the spreading of fire. Moreover, between two consecutive Burner turns fire spreads twice in $G[H]$. Hence, if we only keep track of the copies of $H$ containing a burning vertex, the resulting process is no faster than a burning process on $G^2$ with one new source for each such pair of rounds.

Let $k=b(G^2)$. By the end of round $2k-3$, at most $k-1$ new sources have been introduced in this projected process. If every copy of $H$ were already burning by then, these sources would give a burning process of $G^2$ of length at most $k-1$, contradicting $b(G^2)=k$. Hence at least one copy of $H$ contains no burning vertex after round $2k-3$, and therefore the game on $G[H]$ lasts for at least $2k-2$ rounds. Thus $\bg(G[H])\geq 2b(G^2)-2$. By Proposition 5 in [12], it then follows that $\bg'(G[H])\geq 2b(G^2)-3$.


    For the upper bound, Burner's strategy is as follows.  Identify an optimal sequence of moves $x_1, x_2, \ldots$ for the burning process on $G^2$, and fix any $v \in V(H)$. In the game on $G[H]$, in Burner's $i$th turn, he plays vertex $(x_i,v)$.  (If this vertex is already burning, then Burner plays any legal move.) 
This ensures that the copy of $G$ corresponding to $v$ burns no later than Burner's $b(G^2)$th turn, which happens in round $2b(G^2)-1$ of the Burner-start game and round $2b(G^2)$ of the Staller-start game. Since $G$ is connected and has at least two vertices, every vertex $u\in V(G)$ has a neighbor $u'$. Hence every vertex in the copy of $H$ corresponding to $u$ is adjacent to the burning vertex $(u',v)$. Thus, after one additional spreading phase, all remaining vertices of $G[H]$ are burned. Consequently, $\bg(G[H])\leq 2b(G^2)$ and $\bg'(G[H])\leq 2b(G^2)+1$.
\end{proof}

\begin{example}
The bounds in Theorem~\ref{lexicographic-product} are sharp when $H$ has a universal vertex. It is enough to consider $H=K_2$.

First, let $G=P_3$. Clearly, $b_g(P_3)=b'_g(P_3)=2$. In the Burner-start game on $P_3[K_2]$, Burner starts in the copy corresponding to the middle vertex of $P_3$, and the whole graph burns in round $2$. Hence, $b_g(P_3[K_2])=2=b_g(P_3)$. In the Staller-start game, Staller starts in a copy corresponding to an endvertex of $P_3$. In round $2$, the middle copy burns completely, while the opposite end copy is still unburned. Burner can play only one vertex of this copy, and therefore one more round is necessary. Thus, $b'_g(P_3[K_2])=3=b'_g(P_3)+1$.

Now let $G=P_4$. We have $b_g(P_4)=2$ and $b'_g(P_4)=3$. In the Burner-start game on $P_4[K_2]$, after the spreading phase of round $2$ at least one copy of $K_2$ is still completely unburned. Staller can play one vertex of this copy, but its other vertex remains unburned, so a third round is necessary. On the other hand, Burner can start in a copy corresponding to a middle vertex of $P_4$ and ensure that the whole graph burns by round $3$. Hence, $b_g(P_4[K_2])=3=b_g(P_4)+1$. Finally, in the Staller-start game, Staller can force at least three rounds by starting in a copy corresponding to an endvertex of $P_4$. Conversely, after any first move of Staller, Burner can ensure that all copies are burned by round $3$. Therefore, $b'_g(P_4[K_2])=3=b'_g(P_4)$.

Consequently, all four bounds in the first part of Theorem~\ref{lexicographic-product} are sharp.

\medskip

The bounds in Theorem~\ref{lexicographic-product} are also sharp when $H$ has no universal vertex.

For the lower bound in the Burner-start game, consider $P_7[P_4]$. We have $b(P_7^2)=3$. Let $v_1,\ldots,v_7$ be the vertices of $P_7$ and let $H_i$ denote the copy of $P_4$ corresponding to $v_i$. Burner starts by playing any vertex of $H_4$. In the spreading phase of round $2$, all vertices of $H_3$ and $H_5$ burn. In round $3$, the fire spreads to all vertices of $H_2$ and $H_6$, and in round $4$ to all vertices of $H_1$ and $H_7$. Thus, regardless of Staller's moves, the whole graph is burned by round $4$. By Theorem~\ref{lexicographic-product}, $b_g(P_7[P_4])\ge 2b(P_7^2)-2=4$, and hence $b_g(P_7[P_4])=4=2b(P_7^2)-2$.

For the lower bound in the Staller-start game, consider $C_{16}[P_4]$. We have $b(C_{16}^2)=4$. Label the vertices of $C_{16}$ cyclically by $v_0,\ldots,v_{15}$, and let $H_i$ denote the copy of $P_4$ corresponding to $v_i$. Suppose that Staller starts in $H_0$, which we may assume by symmetry. Burner then plays any vertex of the opposite copy $H_8$ in round $2$. By the spreading phase of round $5$, the fire originating from Staller's first move has reached the nine copies corresponding to vertices at distance at most $4$ from $v_0$ in $C_{16}$. Similarly, the fire originating from Burner's move has reached the seven remaining copies, corresponding to vertices at distance at most $3$ from $v_8$. Hence, the whole graph is burned by round $5$, regardless of Staller's subsequent moves. Theorem~\ref{lexicographic-product} gives $b'_g(C_{16}[P_4])\ge 2b(C_{16}^2)-3=5$, and consequently $b'_g(C_{16}[P_4])=5=2b(C_{16}^2)-3$.

For the upper bound in the Burner-start game, consider $P_6[P_4]$. Since $b(P_6^2)=2$, Theorem~\ref{lexicographic-product} gives $b_g(P_6[P_4])\le 4$. We show that three rounds do not suffice. Let $H_i$ denote the copy of $P_4$ corresponding to $v_i\in V(P_6)$, and suppose that Burner starts in $H_i$. Since $P_4$ has no universal vertex, after the spreading phase of round $2$ there is still an unburned vertex in $H_i$. Staller plays such a vertex, and hence her move does not help the fire to reach any new copy of $P_4$. By the spreading phase of round $3$, the fire originating from Burner's first move can therefore reach only copies $H_j$ for which $d_{P_6}(v_i,v_j)\le 2$. Since for every $v_i\in V(P_6)$ there exists a vertex $v_j$ with $d_{P_6}(v_i,v_j)\ge 3$, the corresponding copy $H_j$ is still completely unburned after the spreading phase of round $3$. Burner can then burn only one vertex of $H_j$ in the selection phase, so the whole graph cannot be burned in round $3$. Hence, $b_g(P_6[P_4])\ge 4$, and consequently $b_g(P_6[P_4])=4=2b(P_6^2)$.

Finally, for the upper bound in the Staller-start game, let $G$ be the tree obtained from three paths of lengths $3$, $2$, and $2$ by identifying one endvertex of each path. Denote their common endvertex by $a$, and denote the three paths by $ab_1b_2b_3$, $ac_1c_2$, and $ad_1d_2$, respectively. Let $H=P_4$. We have $b(G^2)=2$: after choosing $a$ as the first source, all vertices except $b_3$ burn in the next round, and $b_3$ can then be chosen as the second source. We next show that $\bg'(G[P_4])\geq 5$. Staller starts by playing an endvertex in the copy of $P_4$ corresponding to $c_2$. Suppose that Burner plays in round $2$ in the copy corresponding to some vertex $x\in V(G)$. Since $d_G(b_3,d_2)=5$, there exists $z\in\{b_3,d_2\}$ such that $d_G(x,z)\geq 3$. Moreover, $d_G(c_2,b_3)=5$ and $d_G(c_2,d_2)=4$. Hence, after the spreading phase of round $3$, the copy of $P_4$ corresponding to $z$ is still completely unburned. Staller plays an endvertex of this copy in round $3$. After the spreading phase of round $4$, the fire originating from Staller's first move cannot yet have reached this copy, since $d_G(c_2,z)\geq 4$, and the fire originating from Burner's move cannot have reached it either, since $d_G(x,z)\geq 3$. Thus, inside this copy the fire has spread only from the endvertex played by Staller in round $3$. Consequently, at least two vertices in this copy are still unburned after the spreading phase of round $4$. Burner can burn at most one of them in the selection phase, so the game lasts for at least five rounds. By Theorem~\ref{lexicographic-product}, $\bg'(G[P_4])\leq 2b(G^2)+1=5$, and hence $\bg'(G[P_4])=5=2b(G^2)+1$.


Consequently, all four bounds in the second part of Theorem~\ref{lexicographic-product} are sharp.
\end{example}

\section{Open problems}\label{sec:questions}

We conclude with a few open problems that arise naturally from our results and indicate some possible directions for further research on the burning game and its variants.

\medskip

One of the basic questions is to better understand the relation between the burning number, the game burning number, and the cooling number. It was shown in \cite{burning-1} that $b(G)=b_g(G)$ whenever $\diam(G)=2$, while equality need not hold for graphs of diameter $3$. This motivates the following problem.

\begin{problem}
Characterize the connected graphs $G$ for which $b(G)=b_g(G)$, and the graphs for which $b_g(G)=\CL(G)$.
\end{problem}

Another natural question concerns the difference between the Burner-start and Staller-start versions of the game. The two game burning numbers differ by at most one \cite{burning-1}, but little is known about when they are equal.

\begin{problem}
Characterize the graphs $G$ for which $b_g(G)=b'_g(G)$.
\end{problem}

Finally, several questions remain open concerning our bounds for graph products. In particular, the sharpness of some of the bounds for Cartesian and strong products is not known, and it would also be interesting to determine the game burning numbers for classical product families such as grids, cylinders, and tori.

\begin{problem}
Determine the sharpness and characterize the equality cases in the
bounds obtained for graph products in this paper.
\end{problem}

\section{Acknowledgement}
The work was initiated at the 2nd Workshop on Games on Graphs on Rogla in June 2024. No AI was used for this research.

N.C. acknowledges partial support by the Slovenian Research and Innovation Agency (research program P1-0404, and research projects N1-0370, and J1-4008).

V.I.C.\ acknowledges the financial support from the Slovenian Research and Innovation Agency (Z1-50003, J1-70045, P1-0297, N1-0285) and the European Union (ERC, KARST, 101071836).

M.J.\ acknowledges the financial support of the Slovenian Research and Innovation Agency (research core funding No.\ P1-0297 and projects N1-0285, N1-0431.

M.M.\ acknowledges the partial financial support of Ministry of Science, Technological Development and Innovation of Republic of Serbia (451-03-33/2026-03/200125 \& 451-03-34/2026-03/200125).

\bibliographystyle{plain}
\bibliography{referencesBG}
\end{document}